\documentclass[11pt,a4paper]{amsart}
\usepackage[T1]{fontenc}
\usepackage[utf8]{inputenc}
\usepackage{lmodern}
\usepackage{microtype}
\usepackage{amsmath,amssymb,mathtools}
\usepackage[a4paper,margin=29mm,headheight=14pt]{geometry}
\usepackage{enumitem}
\usepackage{booktabs}
\usepackage{needspace}
\usepackage{hyperref}
\hypersetup{hidelinks,pdftitle={Categorical wall-crossing for higher-framed DT/PT quivers},pdfauthor={Wu-yen Chuang}}
\setlist[enumerate]{label=\textnormal{(\roman*)},leftmargin=2.1em,itemsep=3pt}
\numberwithin{equation}{section}
\newtheorem{theorem}{Theorem}[section]
\newtheorem{proposition}[theorem]{Proposition}
\newtheorem{lemma}[theorem]{Lemma}
\newtheorem{corollary}[theorem]{Corollary}
\theoremstyle{definition}
\newtheorem{definition}[theorem]{Definition}
\newtheorem{example}[theorem]{Example}

\theoremstyle{remark}
\newtheorem{remark}[theorem]{Remark}
\newcommand{\C}{\mathbb C}
\newcommand{\Z}{\mathbb Z}
\newcommand{\R}{\mathbb R}
\newcommand{\Q}{\mathbb Q}
\newcommand{\A}{\mathbb A}
\newcommand{\PP}{\mathbb P}

\newcommand{\cO}{\mathcal O}
\newcommand{\cX}{\mathcal X}
\newcommand{\cE}{\mathcal E}
\newcommand{\cF}{\mathcal F}
\newcommand{\cM}{\mathcal M}
\newcommand{\cD}{\mathcal D}
\newcommand{\cS}{\mathcal S}
\newcommand{\Db}{D^{\mathrm b}}
\newcommand{\Perf}{\operatorname{Perf}}
\newcommand{\MF}{\operatorname{MF}}
\newcommand{\GL}{\operatorname{GL}}
\newcommand{\End}{\operatorname{End}}
\newcommand{\Hom}{\operatorname{Hom}}
\newcommand{\RHom}{\operatorname{RHom}}

\newcommand{\Tr}{\operatorname{Tr}}
\newcommand{\Crit}{\operatorname{Crit}}
\newcommand{\Tot}{\operatorname{Tot}}
\newcommand{\Quot}{\operatorname{Quot}}

\newcommand{\id}{\operatorname{id}}
\newcommand{\thick}{\operatorname{thick}}
\newcommand{\gr}{\mathrm{gr}}
\newcommand{\semistable}{\mathrm{ss}}
\newcommand{\pt}{\mathrm{pt}}

\newcommand{\Hall}{\mathsf H}
\newcommand{\boxt}{\mathbin{\boxtimes}}
\newcommand{\sig}{\sigma}
\newcommand{\bZ}{\mathbf Z}
\newcommand{\sod}[1]{\left\langle #1\right\rangle}
\DeclareMathOperator*{\argmin}{arg\,min}
\allowdisplaybreaks[2]
\title[Higher-framed DT/PT quivers]{Categorical wall-crossing for higher-framed DT/PT quivers}
\author{Wu-yen Chuang}
\email{wychuang@gmail.com}
\date{October 19, 2026}

\begin{document}
\begin{abstract}
We study a framed three-loop quiver with $a+r$ incoming arrows and $a$ outgoing arrows. For a generic window parameter, we construct a semiorthogonal decomposition of the bounded derived category of the positive GIT quotient into products of unframed window categories and categories of the negative quotient. The arrow imbalance $r$ determines an interval of admissible slopes of length $r$. The proof separates a lower-convex-hull decomposition of dominant weights, Koszul induction, full faithfulness and semiorthogonality of Hall functors, and the two window equivalences. We extend the decomposition to trace potentials satisfying block additivity. Applications include a precisely specified open chart of the framed conifold quiver, an algebraically defined extended ADHM critical model, and the critical presentation of the Quot scheme of points on $\A^3$. We do not identify the extended ADHM model with moduli of higher-rank pairs having a generically rank-one reduced-curve sheaf. A formal MacMahon identity is proved separately, and the localized categorical rank formula is deduced using the localized strict-slope product formula of P\u{a}durariu--Toda.
\end{abstract}
\maketitle

\section{Introduction and the main statement}\label{sec:intro}

The local categorical DT/PT correspondence of P\u{a}durariu and Toda expresses a DT category as a semiorthogonal decomposition whose factors consist of point contributions and a PT category \cite{PTlocal}. Their construction combines windows for a framed three-loop quiver with categorical Hall products. The point factors are the window and quasi-BPS categories studied in \cite{PTI,PTII}. The corresponding geometric theory for local surfaces is developed in \cite{PTsurfaces}.

The question considered here is the effect of replacing one additional incoming arrow by $r$ additional incoming arrows. This is a question about an asymmetric quiver representation before it is a question about a geometric moduli space. Keeping these two questions separate is essential: an increase in the dimension of a framing space does not, by itself, preserve either the determinant of a surface complex or the multiplicity of a sheaf along a curve.

Fix integers $a\geq0$, $r\geq1$, and $d\geq0$, and put $V=\C^d$. Define
\begin{equation}\label{eq:representation-intro}
 R^{a,r}_d=V^{\oplus(a+r)}\oplus(V^\vee)^{\oplus a}
                 \oplus\End(V)^{\oplus3},
 \qquad \cX^{a,r}_d=[R^{a,r}_d/\GL(V)].
\end{equation}
The framing vertex has dimension one and is not divided out. Equivalently, one may package several arrows using fixed framing spaces and fixed bases. Write
\begin{equation}\label{eq:git-intro}
 I^{a,r}(d)=[(R^{a,r}_d)^{\det\text{-}\semistable}/\GL(V)],\qquad
 P^{a,r}(d)=[(R^{a,r}_d)^{\det^{-1}\text{-}\semistable}/\GL(V)].
\end{equation}
The letters $I$ and $P$ refer to the two quiver chambers. A geometric DT or PT interpretation requires a separate argument in each application.

Let $\cM(n)_w$ be the unframed window category defined in Section~\ref{sec:notation}. Our main ordinary derived-category statement is as follows.

\Needspace{12\baselineskip}
\begin{theorem}\label{thm:main}
Let $\mu\in\R$ satisfy
\begin{equation}\label{eq:generic}
                    2\mu\ell\notin\Z\qquad(1\leq\ell\leq d).
\end{equation}
There is a semiorthogonal decomposition
\begin{equation}\label{eq:main-sod}
 \Db(I^{a,r}(d))=
 \sod{\Phi_{\boldsymbol d,\boldsymbol w,f}
 \left(\left(\mathop{\boxtimes}_{i=1}^k\cM(d_i)_{w_i}\right)
                 \boxt\Db(P^{a,r}(f))\right)}.
\end{equation}
Here $f\geq0$, $d_i>0$, $f+\sum_i d_i=d$, and $w_i\in\Z$. Set
\begin{equation}\label{eq:vi-intro}
 v_i=w_i+d_i\left(f+\sum_{j>i}d_j-\sum_{j<i}d_j\right).
\end{equation}
The indexing data satisfy
\begin{equation}\label{eq:slopes-intro}
 -r-\mu-\frac a2<\frac{v_1}{d_1}<\cdots<\frac{v_k}{d_k}
                                      <-\mu-\frac a2.
\end{equation}
The empty sequence $k=0$, $f=d$, is allowed. The functors $\Phi$ are the Hall functors followed by restriction to the positive chamber, with the last factor lifted through the negative window at parameter $\mu-(d-f)$. Their order is specified in Definition~\ref{def:order}.
\end{theorem}

The genericity condition is imposed throughout, including when $a=0$. Without it, coincident endpoint slopes require additional conventions, and even a weight decomposition can cease to be unique. The case $d=0$ is interpreted as $\Db(\pt)$.

The proof of Theorem~\ref{thm:main} has four logically distinct parts. The lower hull of a sequence of partial sums determines the unframed blocks and the framed tail. A Koszul--Borel--Weil--Bott calculation shows that the Hall images lie in the desired window and generate it. A separate weight-separation argument proves full faithfulness and semiorthogonality. Finally, grade-restriction windows identify the two window categories with the GIT quotients. In particular, we do not deduce semiorthogonality from generation, or pull a semiorthogonal decomposition through a vector-bundle projection without checking morphisms.

For a trace potential whose restriction to the unframed quiver is $\Tr C[A,B]$, block additivity makes the Hall kernels compatible with the potential. The matrix-factorization statement appears in Theorem~\ref{thm:mf}. We then give three applications with carefully delimited geometric scope. The conifold application uses dimension vector $(r,n+r,n)$ and the open condition $(u,b_2):U\oplus V_2\xrightarrow{\sim}V_1$. The extended ADHM application begins with a derived zero locus defined by explicit equations, rather than an asserted identification with a pair moduli stack. The degree-zero application identifies a classical critical locus with $\Quot_{\A^3}(\cO^{\oplus r},d)$.

\subsection*{Changes from the earlier arXiv version}
This paper is a substantially revised version of the arXiv preprint entitled \emph{The higher rank local categorical DT/PT correspondence} \cite{Chuang2023v2}. The quiver-level categorical wall-crossing theorem is retained as the central result, but the exposition and proof have been substantially reorganized and expanded. The genericity hypothesis is now imposed consistently, and the arguments for the weight decomposition, generation, full faithfulness, semiorthogonality, and GIT window equivalences are treated separately; the compatibility assumptions required for the matrix-factorization statement are also made explicit. Several geometric identifications asserted in Section~2 of the earlier version are not valid in the stated higher-rank generality and have been removed. In particular, the conifold construction is replaced by an explicit open-chart calculation for dimension vector $(r,n+r,n)$, subject to the condition that $(u,b_2):U\oplus V_2\to V_1$ is an isomorphism, without the previously asserted identification with a moduli space of perverse coherent systems. The extended ADHM construction is retained as an algebraically defined derived zero locus and critical model, but its identification with the stated moduli of higher-rank pairs supported on reduced curves is no longer claimed. Additional examples and an application to Quot schemes of points on $\A^3$ are included, and the formal MacMahon identity is distinguished from its localized categorical $K$-theoretic interpretation.

\subsection*{Acknowledgments}
The author thanks Emanuel Diaconescu for helpful discussions. The author was partially supported by Taiwan NSTC and NTU TIMS grants.

\subsection*{AI disclosure}
The quiver-level categorical wall-crossing statement was already obtained in the author's 2023 preprint~\cite{Chuang2023v2}. In preparing the present revision, the author used GPT-6 to assist with reviewing mathematical arguments, identifying problems in the earlier geometric identifications, developing the expanded proof arguments, and \LaTeX{} editing. The author takes full responsibility for the mathematical content and claims in this work.

\section{Quivers, weights, windows, and Hall functors}\label{sec:notation}

All stacks and schemes are over $\C$. We write $\Db$ for the bounded derived category of coherent sheaves. Derived categories carry their natural dg enhancements, and pullbacks and pushforwards are derived unless otherwise indicated. Full subcategories generated by specified objects are understood to be thick and idempotent complete. Products of ordinary dg categories are idempotent-completed derived tensor products; for geometric windows they agree with the subcategories generated by external products. Matrix-factorization products are specified separately in Section~\ref{sec:mf}.

\subsection{Weight conventions}
Let $T_d\subset G_d=\GL_d$ be the diagonal torus, and let $\beta_1,\ldots,\beta_d$ be its coordinate characters. We use the positive roots $\beta_i-\beta_j$ for $j<i$ and the dominance convention
\[
 \chi=(\chi_1,\ldots,\chi_d),\qquad \chi_1\leq\cdots\leq\chi_d.
\]
Thus
\begin{equation}\label{eq:rho}
 \rho_d=\frac12\sum_{j<i}(\beta_i-\beta_j),\qquad
 (\rho_d)_j=j-\frac{d+1}{2},\qquad \sig_d=\sum_{j=1}^d\beta_j.
\end{equation}
We write $\Gamma_d(\chi)$ for the irreducible representation of highest weight $\chi$ in this convention. The dot-dominant conjugate of $\eta$ is
\[
 \eta^+=w(\eta+\rho_d)-\rho_d,
\]
where $w(\eta+\rho_d)$ has increasing coordinates. If two coordinates of $\eta+\rho_d$ coincide, this expression denotes a singular term, which is omitted from Borel--Weil--Bott resolutions.

Define the centered root zonotope
\begin{equation}\label{eq:zonotope}
 \bZ_d=\frac32\sum_{i,j=1}^d[0,\beta_i-\beta_j]
      =\sum_{j<i}\left[-\frac32(\beta_i-\beta_j),
                           \frac32(\beta_i-\beta_j)\right].
\end{equation}
For real numbers $l\leq u$, the notation $[l,u]^d$ means
$\sum_j[l\beta_j,u\beta_j]$.

\subsection{Window categories}
In the definitions below, $\chi$ ranges over integral dominant weights and $\mu,\nu\in\R$.
Let $\cX_d=[\End(\C^d)^{\oplus3}/G_d]$. For $w\in\Z$, set
\begin{equation}\label{eq:M-def}
 \cM(d)_w=
 \thick\left\{\cO_{\cX_d}\otimes\Gamma_d(\chi):
 \chi+\rho_d\in\bZ_d+\frac wd\sig_d\right\}.
\end{equation}
Every object has central weight $w$, because scalar matrices act trivially on $\End(\C^d)^{\oplus3}$.

We shall also use an unframed window of arbitrary integral width:
\begin{equation}\label{eq:D-def}
 \cD_r(d;\nu)=
 \thick\left\{\cO_{\cX_d}\otimes\Gamma_d(\chi):
             \chi+\rho_d+\nu\sig_d\in\bZ_d+[-r,0]^d\right\}.
\end{equation}
Its central-weight-$w$ summand is denoted $\cD_r(d;\nu)_w$.

On the asymmetric stack $\cX^{a,r}_d$, define
\begin{align}
 \cE^{a,r}_d(\mu)
 &=\thick\left\{\cO\otimes\Gamma_d(\chi):
 \chi+\rho_d+\mu\sig_d\in
             \bZ_d+[-a/2-r,a/2]^d\right\},\label{eq:E-def}\\
 \cF^{a,r}_d(\mu)
 &=\thick\left\{\cO\otimes\Gamma_d(\chi):
 \chi+\rho_d+\mu\sig_d\in
             \bZ_d+[-a/2,a/2]^d\right\}.\label{eq:F-def}
\end{align}
These are finite sets of generating representations. Under \eqref{eq:generic}, none of the relevant lattice weights lies on a translated bounding facet. Consequently closed and suitably half-open conventions give the same categories. The small window is zero for $a=0$ and $d>0$: summing the coordinates in \eqref{eq:F-def} would give $\sum_j\chi_j=-\mu d\in\Z$, contradicting genericity.

For $d=0$, all framed stacks and windows are $\pt$ and $\Perf(\pt)$, respectively; $\cD_r(0;\nu)_0=\Perf(\pt)$ and the other central summands are zero. We never use an unframed block of dimension zero.

\subsection{The attracting correspondence}
Let $R$ be one of the representations above, let $G=G_d$, and let $\lambda$ be a cocharacter. Write $R^0=R^\lambda$, $R^{\geq0}=R^{\lambda\geq0}$, $L=G^\lambda$, and $P=G^{\lambda\geq0}$. There is a correspondence
\begin{equation}\label{eq:hall-correspondence}
 Z=[R^0/L]\xleftarrow{\ q\ }Y=[R^{\geq0}/P]
                           \xrightarrow{\ p\ }X=[R/G].
\end{equation}
The morphism $p$ is representable and proper: its fibers are closed loci in the appropriate flag variety. The morphism $q$ has finite Tor dimension. The Hall functor is
\begin{equation}\label{eq:Hall}
                         \Hall_\lambda=Rp_*Lq^*.
\end{equation}
For an ordered decomposition $d=d_1+\cdots+d_k+f$, choose strictly decreasing positive weights on the first $k$ blocks and weight zero on the last block. Then
\[
 Z=\prod_{i=1}^k\cX_{d_i}\times\cX^{a,r}_f.
\]
The framing remains on the last block. The attracting stack and its Hall functor depend on the ordered blocks, not on the particular strictly decreasing positive integers chosen. Iterated Hall products agree with the Hall functor for the combined flag, by the correspondence of refined flags and proper base change.

\section{Lower hulls and the canonical weight decomposition}\label{sec:hulls}

For a dominant weight $\chi$, put
\begin{equation}\label{eq:prefix}
 y_j(\chi)=\chi_j+d+1-2j,\qquad
 P_k(\chi)=\sum_{j=1}^k y_j(\chi),\qquad P_0(\chi)=0.
\end{equation}
Notice that $P_d(\chi)=\sum_j\chi_j$.

\begin{lemma}[Subset and prefix inequalities]\label{lem:prefix}
For $l\leq u$, the condition
\[
                  \chi+\rho_d+\mu\sig_d\in\bZ_d+[l,u]^d
\]
is equivalent to
\begin{equation}\label{eq:prefix-bounds}
 P_k(\chi)\geq k(l-\mu),\qquad
 P_d(\chi)-P_{d-k}(\chi)\leq k(u-\mu)
                       \quad(1\leq k\leq d).
\end{equation}
\end{lemma}
\begin{proof}
For a subset $J\subset\{1,\ldots,d\}$ of size $k$, the minimum and maximum of the sum of the $J$-coordinates on $\bZ_d$ are $-3k(d-k)/2$ and $3k(d-k)/2$. Thus the required inequalities for a vector $x$ are
\begin{equation}\label{eq:subset}
 kl-\frac32k(d-k)\leq\sum_{j\in J}x_j
                  \leq ku+\frac32k(d-k).
\end{equation}
These inequalities characterize $\bZ_d+[l,u]^d$, not merely its projections. Indeed, its support function at $c\in\R^d$ is
\[
 \frac32\sum_{i<j}|c_i-c_j|
       +u\sum_{c_i>0}c_i+l\sum_{c_i<0}c_i.
\]
Decompose $c$ into positive multiples of indicators of its positive upper level sets and negative multiples of indicators of its negative lower level sets. Adding the corresponding inequalities in \eqref{eq:subset} gives this support-function bound. Separation of convex sets proves sufficiency.

The coordinates of $\chi+\rho_d+\mu\sig_d$ are strictly increasing. The extremal subsets are therefore the first and last $k$ indices. Finally,
\[
 \sum_{j=1}^k(\chi_j+(\rho_d)_j+\mu)
       =P_k(\chi)-\frac32k(d-k)+k\mu,
\]
and the analogous last-$k$ identity gives \eqref{eq:prefix-bounds}.
\end{proof}

Let $H_\chi:[0,d]\to\R$ be the lower convex hull of the points $(k,P_k(\chi))$. Consecutive collinear pieces are merged. Equivalently, $H_\chi$ is the greatest convex function lying below these points. Its segment slopes are strictly increasing. The inequalities \eqref{eq:prefix-bounds} say exactly that all these slopes lie in $[l-\mu,u-\mu]$: the first slope is $\min_{k>0}P_k/k$, and the last is $\max_{k<d}(P_d-P_k)/(d-k)$.

\begin{lemma}[A single block]\label{lem:block}
Suppose $p<q$ are consecutive vertices of $H_\chi$. Set $n=q-p$,
$\chi_{[p,q]}=(\chi_{p+1},\ldots,\chi_q)$, and $w=\sum_{j=p+1}^q\chi_j$. The slope of this segment is
\begin{equation}\label{eq:segment-slope}
                    s=\frac wn+d-p-q.
\end{equation}
Moreover,
\begin{equation}\label{eq:block-window}
        \chi_{[p,q]}+\rho_n\in\bZ_n+\frac wn\sig_n.
\end{equation}
Conversely, \eqref{eq:block-window} is equivalent to the intermediate points of this block lying on or above its chord.
\end{lemma}
\begin{proof}
Summing \eqref{eq:prefix} from $p+1$ to $q$ gives \eqref{eq:segment-slope}. The local partial sums for the block differ from the global ones by the linear function $k(d-p-q)$. The condition of lying above the chord is consequently the lower-prefix inequality of Lemma~\ref{lem:prefix} with $l=u=w/n$. The upper-suffix inequalities follow from the fixed total sum $w$.
\end{proof}

Set, for the rest of the framed argument,
\begin{equation}\label{eq:LTU}
             L=-r-\mu-a/2,\qquad T=-\mu-a/2,\qquad U=-\mu+a/2.
\end{equation}

\begin{proposition}[Canonical decomposition]\label{prop:decomposition}
Assume \eqref{eq:generic}, and let $\chi$ be a generator weight of $\cE^{a,r}_d(\mu)$. There is a unique decomposition into consecutive dominant blocks
\[
                 \chi=\chi_1+\cdots+\chi_k+\chi',
\]
where $\dim\chi_i=d_i>0$, $\dim\chi'=f$, and $e=\sum_i d_i=d-f$, such that
\begin{align}
 \chi_i+\rho_{d_i}&\in\bZ_{d_i}+\frac{w_i}{d_i}\sig_{d_i},
       &w_i&=\sum\chi_i,\label{eq:canonical-unframed}\\
 \chi'+\rho_f+(\mu-e)\sig_f
      &\in\bZ_f+[-a/2,a/2]^f,\label{eq:canonical-tail}
\end{align}
and \eqref{eq:slopes-intro} holds with \eqref{eq:vi-intro}.
\end{proposition}
\begin{proof}
Take the initial segments of $H_\chi$ whose slopes are less than $T$. Their endpoints give the first $k$ blocks; the remaining coordinates give $\chi'$. By Lemma~\ref{lem:block}, the slope of the $i$th segment is $v_i/d_i$, and \eqref{eq:canonical-unframed} follows. A hull slope is a rational number with denominator at most the length of its segment. Equality with any of $L,T,U$ would imply $2\mu\ell\in\Z$ for a positive $\ell\leq d$. Hence all necessary endpoint inequalities are strict.

For a tail coordinate $j=e+h$, one has
\[
 y_{e+h}(\chi)=y_h(\chi')-e.
\]
Its local hull slopes are therefore the remaining global slopes plus $e$, and lie between $T+e$ and $U+e$. Lemma~\ref{lem:prefix} gives \eqref{eq:canonical-tail}.

Conversely, \eqref{eq:canonical-unframed} makes each first block lie above its chord. The strictly increasing slopes in \eqref{eq:slopes-intro} make these chords initial segments of a convex polygon. By \eqref{eq:canonical-tail}, the slopes of the tail hull are greater than $T$ in global coordinates. Thus these segments are exactly the segments of $H_\chi$ below $T$. Uniqueness follows from uniqueness of the lower hull.
\end{proof}

There is a useful coarser version. Define
\begin{equation}\label{eq:p-function}
 p_k(\chi)=P_k(\chi)-kT,\qquad
 p(\chi)=\min_{0\leq k\leq d}p_k(\chi),\qquad
 e(\chi)=\argmin_{0\leq k\leq d}p_k(\chi).
\end{equation}
The minimizer is unique: a tie between $k$ and $k'$ would contradict \eqref{eq:generic}. It is the total length $e$ of the initial hull segments below $T$. If $w$ is the sum of the first $e$ coordinates, then
\begin{equation}\label{eq:coarse-p}
             p(\chi)=p(e,w):=w+e(d-e)+e(\mu+a/2).
\end{equation}
For $e>0$ this number is negative. The case $e=0$ is precisely the small window and has $p=0$.

\begin{corollary}\label{cor:coarse-decomp}
Put $f=d-e$, $\nu_e=\mu+f+a/2$, and $\nu_f=\mu-e$. The coarse blocks of $\chi$ are generator weights of
\begin{equation}\label{eq:coarse-source}
                \cD_r(e;\nu_e)_w\boxt\cF^{a,r}_f(\nu_f).
\end{equation}
Conversely, concatenating any generator weights of a nonzero category \eqref{eq:coarse-source} gives a globally dominant generator weight of $\cE^{a,r}_d(\mu)$, whose unique coarse type is $(e,w)$.
\end{corollary}
\begin{proof}
The prefix has all its global hull slopes in $(L,T)$, and passing to its local coordinates subtracts $f$. This gives the interval $(-r-\nu_e,-\nu_e)$, as required by \eqref{eq:D-def}. The assertion about the tail is Proposition~\ref{prop:decomposition}.

For the converse, the last coordinate of the prefix is at most $T+e-f-1$, whereas the first coordinate of a nonempty tail is at least $T+e-f+1$. These follow from the one-coordinate subset bounds of Lemma~\ref{lem:prefix} in the two blocks. Thus concatenation is dominant. Its global hull slopes are below $T$ in the prefix and above $T$ in the tail. This also proves its membership and uniqueness of type.
\end{proof}

\begin{example}\label{ex:weights}
Take $d=2$, $a=1$, $r=2$, $\mu=1/7$. Then
$L=-37/14$, $T=-9/14$, and $U=5/14$. The weights $(0,0)$, $(-2,1)$, $(-3,0)$, and $(-2,-2)$ have, respectively: only a framed tail; one unframed block with slope $-1$ and a tail of length one; two unframed blocks with slopes $-2,-1$; and one unframed block of length two and slope $-2$. In the second case $w_1=-2$ but $v_1=-1$, illustrating the importance of the shift in \eqref{eq:vi-intro}.
\end{example}

\section{Koszul induction and separation of weights}\label{sec:koszul}

We record the categorical input used in both the unframed and framed arguments. Its first part is the standard Koszul--Borel--Weil--Bott calculation; see \cite[Section 3.2]{HLS} and \cite[Proposition 2.1]{PTlocal}. It applies to an arbitrary representation, not only to a symmetric one.

\begin{lemma}[The Hall resolution]\label{lem:koszul}
Use the correspondence \eqref{eq:hall-correspondence}, and suppose that $\lambda$ has decreasing weights on its ordered blocks. Let $\mathcal N_\lambda$ be the multiset of weights of $R^{\lambda<0}$. For an $L$-dominant weight $\chi$, the object
\[
              \Hall_\lambda(\cO_Z\otimes\Gamma_L(\chi))
\]
has a finite resolution whose terms, with their cohomological shifts and multiplicities, are
\begin{equation}\label{eq:koszul-terms}
 \cO_X\otimes\Gamma_d\bigl((\chi-\sig_I)^+\bigr),\qquad
 \sig_I=\sum_{\gamma\in I}\gamma,
                    \quad I\subset\mathcal N_\lambda.
\end{equation}
Singular terms are omitted. If $\chi$ is also $G$-dominant, there is a natural map
\begin{equation}\label{eq:leading-map}
 \cO_X\otimes\Gamma_d(\chi)
       \longrightarrow\Hall_\lambda(\cO_Z\otimes\Gamma_L(\chi))
\end{equation}
whose cone is generated by the terms with $I\neq\varnothing$.
\end{lemma}
\begin{proof}
Resolve the closed embedding $R^{\geq0}\hookrightarrow R$ by its equivariant Koszul complex. The exterior powers of the conormal representation contribute the weights $-\sig_I$. Induction from $P$ to $G$ is computed on the flag variety. Refining to a Borel subgroup and applying Borel--Weil--Bott gives \eqref{eq:koszul-terms}; a regular term has shift $|I|-\ell(w_I)$, where $w_I$ makes $\chi-\sig_I+\rho_d$ dominant. A singular term has zero derived induction. The augmentation gives \eqref{eq:leading-map}, and the augmented Koszul filtration has precisely the nonempty terms as its remaining graded pieces. This proof uses only the negative normal representation, so does not require symmetry of $R$.
\end{proof}

The following separation principle explains why the inequalities established below imply full faithfulness rather than only generation.

\begin{lemma}[Weight separation]\label{lem:separation}
Use \eqref{eq:hall-correspondence}, with $\lambda$ decreasing on the ordered blocks as in Lemma~\ref{lem:koszul}. The cocharacter $\lambda$ acts trivially on $R^0$. Let $\mathcal A\subset\Db(Z)$ be generated by bundles $\cO_Z\otimes\Gamma_L(\chi)$ on which $\lambda$ has the same weight $b$. Suppose these $\chi$ are $G$-dominant and every nonsingular nonempty term of their Hall resolutions satisfies
\begin{equation}\label{eq:self-separation}
                         \langle\lambda,(\chi-\sig_I)^+\rangle>b.
\end{equation}
Then $\Hall_\lambda|_{\mathcal A}$ is fully faithful.

More generally, let $F\in\Db(X)$ have a finite resolution by bundles $\cO_X\otimes\Gamma_d(\xi)$ with
$\langle\lambda,\xi\rangle>b$. Then
\begin{equation}\label{eq:cross-separation}
                  \RHom_X(F,\Hall_\lambda(A))=0
                          \qquad(A\in\mathcal A).
\end{equation}
\end{lemma}
\begin{proof}
The coordinate functions of the strictly positive part of $R^{\geq0}$ have negative $\lambda$-weights. The Lie algebra of the unipotent radical of $P$ has positive weights, so its Chevalley--Eilenberg cochains have nonpositive weights, strictly negative in positive exterior degree. Consequently, for two objects $A,B$ pulled back from $Z$ with the same pure $\lambda$-weight,
\begin{equation}\label{eq:q-weight-zero}
           \RHom_Y(q^*B,q^*A)=\RHom_Z(B,A).
\end{equation}
Indeed, the weight-zero part of the complex computing the left side consists exactly of the constant functions in the positive directions and exterior degree zero for the unipotent radical. The remaining equivariance is that of $L$.

For a $G$-dominant weight $\xi$, the minimum $\lambda$-weight of $\Gamma_d(\xi)$ is $\langle\lambda,\xi\rangle$. Its restriction to $P$ has a filtration by $\lambda$-weights. If this minimum exceeds $b$, the same cochain calculation gives
\[
 \RHom_Y(p^*(\cO_X\otimes\Gamma_d(\xi)),q^*A)=0.
\]
This proves \eqref{eq:cross-separation} by adjunction and passage through a finite resolution.

For full faithfulness, consider the counit
\[
                  p^*p_*q^*B\longrightarrow q^*B.
\]
On a bundle generator, its lowest-weight part is the projection of $\Gamma_d(\chi)|_P$ to the $L$-representation $\Gamma_L(\chi)$ of weight $b$. All other weights of this restriction exceed $b$. By Lemma~\ref{lem:koszul} and \eqref{eq:self-separation}, all the other terms in the filtered cone have generator weights strictly greater than $b$. The preceding vanishing therefore annihilates this cone against $q^*A$. Together with \eqref{eq:q-weight-zero} and adjunction, this gives
\[
 \RHom_X(\Hall_\lambda(B),\Hall_\lambda(A))
                         \simeq\RHom_Z(B,A).
\]
The argument extends from generators to their thick closures. Here weights are weights of the bundle generators on the fixed locus; functions on the attracting locus are accounted for by the cochain calculation, rather than being incorrectly treated as weight zero.
\end{proof}

\section{Unframed windows of arbitrary width}\label{sec:unframed}

We first prove the unframed statement needed to refine the coarse framed decomposition. This is a finite-window version of the Hall decompositions underlying \cite{PTI}; the argument below makes the dependence on the interval width explicit.

For ordered positive integers $n_1,\ldots,n_k$ of total $n$ and integers $w_i$, put
\begin{equation}\label{eq:unframed-slopes}
       s_i=\frac{w_i}{n_i}+\sum_{j>i}n_j-\sum_{j<i}n_j.
\end{equation}
Let $H_S$ be the polygon starting at $(0,0)$ with consecutive horizontal lengths $n_i$ and slopes $s_i$. Its endpoint is $(n,\sum_iw_i)$.

\begin{lemma}[Unframed corrections]\label{lem:unframed-corrections}
Suppose $s_1<\cdots<s_k$. Choose generator weights $\chi_i$ of $\cM(n_i)_{w_i}$ and concatenate them to $\chi$. Then $\chi$ is $G_n$-dominant and $H_\chi=H_S$. For every nonsingular weight $\xi=(\chi-\sig_I)^+$ in the Hall resolution,
\begin{equation}\label{eq:hull-monotone}
                               H_\xi\geq H_S.
\end{equation}
If $I\neq\varnothing$, the inequality is strict at at least one internal vertex of $H_S$. If all $s_i\in[-r-\nu,-\nu]$, every such $\xi$ belongs to $\cD_r(n;\nu)$.
\end{lemma}
\begin{proof}
The one-coordinate bounds in a block are
\[
 \frac{w_i}{n_i}-(n_i-1)\leq\chi_{i,1}
                  \leq\chi_{i,n_i}\leq\frac{w_i}{n_i}+(n_i-1).
\]
The difference of adjacent block means is $s_{i+1}-s_i+n_i+n_{i+1}$. Thus the first coordinate of the next block exceeds the last coordinate of the preceding block by at least $s_{i+1}-s_i+2$. Concatenation is dominant, and Lemma~\ref{lem:block} identifies its hull.

Write $\chi_i+\rho_{n_i}=z_i+(w_i/n_i)\sig_{n_i}$ with $z_i\in\bZ_{n_i}$. Directly comparing the block and ambient $\rho$-shifts gives
\begin{equation}\label{eq:unframed-coeff}
 \chi+\rho_n=\sum_i z_i
 +\frac32\sum_{\substack{j\text{ in an earlier block}\\h\text{ in a later block}}}
       (\beta_h-\beta_j)+\sum_i s_i\sig_{n_i}.
\end{equation}
The negative normal weights are three copies of each displayed cross-block root. Subtracting $\sig_I$ changes its coefficient from $3/2$ to $3/2-m_{hj}$, with $0\leq m_{hj}\leq3$. The result remains in $\bZ_n+\sum_i s_i\sig_{n_i}$. Weyl conjugation preserves the larger region $\bZ_n+[-r-\nu,-\nu]^n$, proving the window assertion.

For the hull assertion, for any real $t$ put
\[
               p_t(\eta)=\min_{0\leq j\leq n}(P_j(\eta)-tj).
\]
The subset bound for the root part of \eqref{eq:unframed-coeff}, after any subtraction and permutation, gives
\begin{equation}\label{eq:dual-hull}
              p_t(\xi)\geq\sum_i n_i\min(s_i-t,0)=p_t(\chi).
\end{equation}
The convex-duality formula $H_\eta(x)=\sup_t(tx+p_t(\eta))$ now proves \eqref{eq:hull-monotone}.

Suppose $I$ is nonempty. Choose a selected cross-block root and a gap $t$ between consecutive slopes which it crosses. Equality in \eqref{eq:dual-hull} would require choosing exactly the union of blocks with $s_i<t$. On that union, the root part of \eqref{eq:unframed-coeff} attains its lower subset bound before subtraction. Subtracting any selected root crossing this union increases its coordinate sum strictly; all other selected roots make a nonnegative change to this sum. Equality is therefore impossible. The hulls are distinct. If they agreed at all the vertices of $H_S$, convexity would force $H_\xi\leq H_S$ on each segment, contrary to distinctness and \eqref{eq:hull-monotone}. Since the endpoints agree, there is strict inequality at an internal vertex.
\end{proof}

\begin{proposition}\label{prop:unframed-sod}
Suppose $2\nu\ell\notin\Z$ for $1\leq\ell\leq n$. For each total weight $w$, the Hall functors give a semiorthogonal decomposition
\begin{equation}\label{eq:unframed-sod}
 \cD_r(n;\nu)_w=
 \sod{\mathop{\boxtimes}_{i=1}^k\cM(n_i)_{w_i}},
\end{equation}
where $\sum_i n_i=n$, $\sum_iw_i=w$, and
\[
                   -r-\nu<s_1<\cdots<s_k<-\nu.
\]
The components are ordered by increasing lexicographic order of
\begin{equation}\label{eq:lex-hull}
                         (H_S(1),\ldots,H_S(n-1)).
\end{equation}
Different central weights are orthogonal.
\end{proposition}
\begin{proof}
\emph{Containment and generation.}
Containment is Lemma~\ref{lem:unframed-corrections}. Every generator weight of the left side has a unique lower hull and hence a unique block decomposition by Lemma~\ref{lem:block}. The map \eqref{eq:leading-map} expresses its bundle in terms of its Hall image and bundles with a strictly larger hull. The finite set of generator weights permits downward induction from maximal hulls. Equivalently, induct on the finitely many values of $\sum_{j=1}^{n-1}H_\chi(j)$, starting with the largest. If there is only one block, the generator already belongs to that block category and there are no cross-block corrections. This proves generation.

\emph{Full faithfulness.}
Fix a type $S$ and a cocharacter $\lambda$ with strictly decreasing constants on its blocks. On the source of its Hall functor every object has the same $\lambda$-weight $b=\sum_i\lambda_iw_i$. At a block endpoint $m_j=n_1+\cdots+n_j$, every source generator has $P_{m_j}(\chi)=H_S(m_j)$. Lemma~\ref{lem:unframed-corrections} implies that every nonempty correction has no smaller prefix sum at these endpoints and has a strictly larger one at some endpoint. Summation by parts gives
\[
 \langle\lambda,\xi-\chi\rangle
  =\sum_{j=1}^{k-1}(\lambda_j-\lambda_{j+1})
                 (P_{m_j}(\xi)-H_S(m_j))>0.
\]
There is no endpoint term because the total weights agree. Lemma~\ref{lem:separation} proves full faithfulness.

\emph{Semiorthogonality.}
Let $S'$ be later than $S$ in \eqref{eq:lex-hull}. Among the internal vertices $m_j$ of $H_S$, the first nonzero difference $H_{S'}(m_j)-H_S(m_j)$ is positive. To see this, use that $H_S$ is affine between its consecutive vertices: if $H_{S'}$ agrees at both ends of such an interval, convexity gives $H_{S'}\leq H_S$ inside it. A first lexicographically positive difference cannot occur in that interval without a positive difference at its right endpoint; a negative difference earlier would contradict the assumed lexicographic order.

Every correction $\xi$ in the resolution of a source generator of type $S'$ has $P_{m_j}(\xi)\geq H_{S'}(m_j)$. Thus the vector
\[
                  (P_{m_j}(\xi)-H_S(m_j))_{j=1}^{k-1}
\]
is lexicographically positive. Choose the positive drops $\lambda_j-\lambda_{j+1}$ successively much smaller as $j$ increases. For example, take descending powers of an integer larger than a bound for all the finitely many differences under consideration. The pairing of every such $\xi$ with $\lambda$ then exceeds the common weight $b$ of the target source category. The attracting correspondence is unchanged by this choice of drops. Lemma~\ref{lem:separation} proves that morphisms from the later Hall image to the earlier one vanish. This proves \eqref{eq:unframed-sod}.
\end{proof}

\section{The asymmetric Hall decomposition}\label{sec:framed}

We now work directly on $\cX^{a,r}_d$. This avoids an unjustified passage of semiorthogonality along the projection that forgets the extra incoming arrows.

\begin{lemma}[Strict increase of the coarse invariant]\label{lem:framed-corrections}
Let $\chi$ be a generator of $\cE^{a,r}_d(\mu)$ of coarse type $(e,w)$, and put $f=d-e$. Use the two-block cocharacter
\[
                   \tau_e(t)=(t\id_{\C^e},\id_{\C^f}).
\]
For every nonsingular term $\xi=(\chi-\sig_I)^+$ of its Hall resolution, $\xi$ remains a generator weight of $\cE^{a,r}_d(\mu)$. If $I\neq\varnothing$, then
\begin{equation}\label{eq:p-strict}
                               p(\xi)>p(\chi).
\end{equation}
\end{lemma}
\begin{proof}
The assertion is vacuous for $e=0$. Otherwise, Corollary~\ref{cor:coarse-decomp} gives an expression
\begin{equation}\label{eq:framed-coeff}
 \chi+\rho_d+\mu\sig_d=z_e+z_f
       +\frac32\sum_{j\leq e<h}(\beta_h-\beta_j)
       +\sum_{j=1}^d c_j\beta_j,
\end{equation}
with $z_e\in\bZ_e$, $z_f\in\bZ_f$, and
\begin{equation}\label{eq:strict-coeff}
 \begin{cases}
 -a/2-r\leq c_j<-a/2,&j\leq e,\\
 -a/2<c_j\leq a/2,&j>e.
 \end{cases}
\end{equation}
Here an empty block is omitted. The strict inequalities can be arranged simultaneously: all the corresponding upper subset inequalities in the first block, and lower subset inequalities in the second, are strict by genericity. There are finitely many of them, so the relevant endpoint of each box can be tightened by a common small positive number before applying Lemma~\ref{lem:prefix}. If $a=0$, a nonempty second block is impossible.

The negative normal weights are exactly
\begin{equation}\label{eq:negative-framed}
 \mathcal N_e=
 \{(\beta_h-\beta_j)^{\times3}:j\leq e<h\}
                       \sqcup\{(-\beta_j)^{\times a}:j\leq e\}.
\end{equation}
In particular, they do not involve the number $r$ of extra incoming arrows. Subtraction of $\sig_I$ changes the cross-root coefficients to numbers in $[-3/2,3/2]$ and changes $c_j$ to $c_j+m_j$, with $0\leq m_j\leq a$ on the first block and $m_j=0$ on the second. All resulting box coordinates belong to $[-a/2-r,a/2]$. Weyl invariance proves containment.

For a subset $J$ corresponding to the first $k$ coordinates after dot-dominant conjugation, the root subset bound gives
\begin{equation}\label{eq:p-bound}
 P_k(\xi)-kT\geq\sum_{j\in J}(c_j+m_j+a/2)
                         \geq\sum_{j\leq e}(c_j+a/2)=p(\chi).
\end{equation}
The last equality follows by summing \eqref{eq:framed-coeff} over the first block. By \eqref{eq:strict-coeff}, equality in the second inequality of \eqref{eq:p-bound} requires $J=\{1,\ldots,e\}$ and $m_j=0$ there. For this subset, equality in the first inequality also requires that no cross root was selected: each selected cross root strictly increases its coordinate sum. Thus equality throughout is possible only for $I=\varnothing$. Taking the minimum over $k$ proves \eqref{eq:p-strict}.
\end{proof}

\begin{proposition}[Coarse semiorthogonal decomposition]\label{prop:coarse-sod}
Assume \eqref{eq:generic}. Put $f=d-e$, $\nu_e=\mu+f+a/2$, and $\nu_f=\mu-e$. The two-block Hall functors are fully faithful on
\[
                    \cD_r(e;\nu_e)_w\boxt\cF^{a,r}_f(\nu_f)
\]
and give a semiorthogonal decomposition of $\cE^{a,r}_d(\mu)$. The nonzero components are ordered by increasing $p(e,w)$ in \eqref{eq:coarse-p}. The component with $e=0$ is the inclusion of $\cF^{a,r}_d(\mu)$ and is last whenever it is nonzero.
\end{proposition}
\begin{proof}
Corollary~\ref{cor:coarse-decomp} and Lemmas~\ref{lem:koszul} and \ref{lem:framed-corrections} show that each image lies in $\cE^{a,r}_d(\mu)$. The same lemmas prove generation by induction on the finite set of values of $p(\chi)$, starting with the largest: the leading map has only corrections with strictly larger $p$. If there are no corrections, the leading map is already an isomorphism.

Fix a coarse type $(e,w)$ with $e>0$. All source objects have pure $\tau_e$-weight $w$. For a nonempty correction $\xi$, one has
\[
 \sum_{j\leq e}\xi_j+e(d-e)-eT
       =p_e(\xi)\geq p(\xi)>p(e,w).
\]
Hence $\sum_{j\leq e}\xi_j>w$. Lemma~\ref{lem:separation} proves full faithfulness. The case $e=0$ is the identity inclusion.

Now let a source type $(e',w')$ be later than $(e,w)$. Every term $\xi$ of its Hall resolution, including the empty term, satisfies
\[
                       p_e(\xi)\geq p(\xi)\geq p(e',w')>p(e,w).
\]
The same separation lemma gives vanishing from its image to the earlier image. Distinct types cannot have equal $p$: an equality with $e\neq e'$ would imply $2\mu|e-e'|\in\Z$, and for $e=e'$ it forces $w=w'$. This proves the proposition.
\end{proof}

We now specify a total order for the refined components. It is part of the statement, rather than an implicit choice.

\begin{definition}[Order of components]\label{def:order}
For an indexing datum $(\boldsymbol d,\boldsymbol w,f)$ in Theorem~\ref{thm:main}, set $e=d-f$ and $w=\sum_iw_i$. First order the data by increasing $p(e,w)$. For a fixed pair $(e,w)$, form the unframed polygon of horizontal lengths $d_i$ and slopes
\[
             \frac{w_i}{d_i}+\sum_{j>i}d_j-\sum_{j<i}d_j
                            =\frac{v_i}{d_i}-f.
\]
Order these polygons increasingly by their values at $1,\ldots,e-1$, lexicographically. If $e\leq1$, the latter ordering is empty. Strictly increasing segment slopes ensure that two identical polygons give the same datum.
\end{definition}

\begin{theorem}[Window decomposition]\label{thm:window-sod}
Under \eqref{eq:generic}, the Hall functors give
\begin{equation}\label{eq:window-sod}
 \cE^{a,r}_d(\mu)=
 \sod{\left(\mathop{\boxtimes}_{i=1}^k\cM(d_i)_{w_i}\right)
                  \boxt\cF^{a,r}_f(\mu-e)},
\end{equation}
with the indexing inequalities \eqref{eq:slopes-intro} and the order of Definition~\ref{def:order}.
\end{theorem}
\begin{proof}
Apply Proposition~\ref{prop:unframed-sod} to each first factor in Proposition~\ref{prop:coarse-sod}. Its parameter is $\nu_e=\mu+f+a/2$, which satisfies the same genericity condition for dimensions at most $e$. Its local slopes are $v_i/d_i-f$, so the resulting interval is exactly \eqref{eq:slopes-intro}. Associativity of the Hall correspondence identifies the iterated functor with the functor of the full flag. Refinement of a semiorthogonal decomposition by fully faithful semiorthogonal decompositions proves \eqref{eq:window-sod} with the stated order.
\end{proof}

\section{The two GIT windows}\label{sec:git}

For the symmetric comparison representations in this section, we also use the definitions of $R^{a,r}_d$, $I^{a,r}(d)$, and $P^{a,r}(d)$ with $r=0$. We supply the dependence on $r$ in the window comparison. The two general inputs are the grade-restriction theorem of Halpern-Leistner \cite{HL} and the magic-window theorem for quasi-symmetric representations of Halpern-Leistner and Sam \cite{HLS}.

\begin{lemma}[Stability and stratum widths]\label{lem:kn}
Positive semistability for $R^{a,r}_d$ is the condition that the incoming vectors generate $V$ under the three loop operators. Negative semistability is the condition that no nonzero loop-invariant subspace is killed by all outgoing covectors. In both chambers semistability equals stability, and the stabilizer is trivial.

The unstable strata may be indexed by two-step cocharacters conjugate to $-\tau_e$ in the positive chamber and $\tau_e$ in the negative chamber, for $1\leq e\leq d$. Their window widths are
\begin{equation}\label{eq:widths}
              \eta_e^+=(a+r)e+2e(d-e),\qquad
              \eta_e^-=ae+2e(d-e).
\end{equation}
Only nonempty strata are included.
\end{lemma}
\begin{proof}
The Hilbert--Mumford criterion for the determinant linearization gives the cyclic condition; for its inverse it gives the cocyclic condition. For the first condition, let $S$ be the smallest invariant subspace containing all incoming vectors. A destabilizing filtration can be replaced by weight zero on $S$ and a single negative weight on its complement. Indeed, positive weights do not improve the destabilizing ratio, and equalizing the negative weights maximizes that ratio by Cauchy--Schwarz. This produces $-\tau_e$, with $e=d-\dim S$, up to conjugacy. Dually, for the negative chamber use the largest invariant subspace contained in the common kernel of the outgoing maps. This gives $\tau_e$. The usual strata are obtained by requiring that these subspaces have the indicated dimensions and satisfy the induced cyclic or cocyclic condition on the complementary block. This argument is the two-step stratification used in \cite[Section 3.7]{PTI}, with the number of framing arrows left arbitrary.

For $-\tau_e$, the negative representation weights consist of $(a+r)e$ incoming directions and $3e(d-e)$ loop directions. Subtracting the $e(d-e)$ negative Lie-algebra directions gives $\eta_e^+$. For $\tau_e$, the outgoing directions contribute $ae$ instead, giving $\eta_e^-$. This proves \eqref{eq:widths}.

An automorphism fixing all loop and framing data is the identity in the cyclic chamber because it is the identity on a generating set. In the cocyclic chamber, the image of its difference from the identity is invariant and killed by all outgoing maps, hence zero. The same criteria exclude strictly semistable points.
\end{proof}

\begin{proposition}[Window equivalences]\label{prop:git-windows}
Assume \eqref{eq:generic}. Restriction induces equivalences
\begin{equation}\label{eq:git-equivalences}
 \cE^{a,r}_d(\mu)\xrightarrow{\sim}\Db(I^{a,r}(d)),\qquad
 \cF^{a,r}_d(\mu)\xrightarrow{\sim}\Db(P^{a,r}(d)).
\end{equation}
\end{proposition}
\begin{proof}
The case $d=0$ is immediate.

\emph{Full faithfulness.}
The effective centers of the positive and negative windows are, respectively,
\[
                       (\mu+r/2)\sig_d\quad\text{and}\quad\mu\sig_d.
\]
For a generator of the large window and every weight $\theta$ of its representation, the root and box inequalities imply
\begin{equation}\label{eq:positive-grade}
 -\eta_e^+/2<\langle-\tau_e,\theta+(\mu+r/2)\sig_d\rangle
                                               <\eta_e^+/2.
\end{equation}
For a generator of the small window they imply
\begin{equation}\label{eq:negative-grade}
 -\eta_e^-/2<\langle\tau_e,\theta+\mu\sig_d\rangle
                                               <\eta_e^-/2.
\end{equation}
One can check these bounds directly on the extremal weights of the representation. For example, after centering a symmetric box of radius $b/2$, the first-$e$ bound of Lemma~\ref{lem:prefix} becomes
\[
 -\frac{be+2e(d-e)}2
       <\sum_{j\leq e}\bigl(\chi_j+\text{center}\bigr),
\]
and the last-$e$ bound gives the opposite extreme. All representation weights lie in the convex hull of the Weyl orbit of $\chi$. Take $b=a+r$ or $b=a$. Equality would force $2\mu e\in\Z$, so the inequalities are strict. By Lemma~\ref{lem:kn}, these bounds put our generators inside the grade-restriction windows for every stratum. The full-faithfulness part of the grade-restriction theorem \cite[Theorem 2.10]{HL} proves full faithfulness of both restrictions.

\emph{Essential surjectivity in the negative chamber.}
For $a>0$, forgetting the additional $r$ incoming arrows gives a vector-bundle projection
\[
                    P^{a,r}(d)\longrightarrow P^{a,0}(d).
\]
Negative stability does not involve those arrows. The representation $R^{a,0}_d$ is self-dual, and its weights are quasi-symmetric and span the weight space. We also check the invariant generic-direction hypothesis of the magic-window theorem. For the symmetric representation $R^{b,0}_d$ with $b>0$, its full weight zonotope is
\[
                         \Sigma_b=2\bZ_d+[-b,b]^d.
\]
The subset-inequality argument of Lemma~\ref{lem:prefix}, with the root coefficient doubled, shows that every facet normal is proportional to $\pm\mathbf1_J$ for a nonempty subset $J\subset\{1,\ldots,d\}$. Its pairing with the Weyl-invariant direction $\sig_d$ is $\pm|J|\neq0$. Thus $\sig_d$ is $\Sigma_b$-generic. The condition \eqref{eq:generic} gives the required avoidance of lattice weights on the translated window boundary. The determinant chamber has stable equals semistable by Lemma~\ref{lem:kn}. Taking $b=a$, the magic-window theorem \cite[Theorem 3.2 and Lemma 2.9]{HLS} therefore says that the bundles defined by the small-window weights generate the base quotient. Their pullbacks generate the total space: for an affine vector-bundle morphism, pullbacks of compact generators generate the derived category of quasi-coherent sheaves, since its pushforward is conservative; taking compact objects gives the assertion for perfect complexes. These quotients are smooth, so perfect and bounded coherent complexes agree. This proves essential surjectivity. When $a=0$ and $d>0$, both the negative quotient and the small window are empty or zero, respectively.

\emph{Essential surjectivity in the positive chamber.}
Adjoin $r$ outgoing arrows. The resulting symmetric representation is $R^{a+r,0}_d$, and its positive quotient is a vector bundle over $I^{a,r}(d)$. The polytope identity
\begin{equation}\label{eq:symmetric-translation}
 \left(\bZ_d+[-a/2-r,a/2]^d\right)+\frac r2\sig_d
             =\bZ_d+[-(a+r)/2,(a+r)/2]^d
\end{equation}
identifies the large-window weights with the symmetric magic window at parameter $\mu+r/2$. The same generic-direction check applies with $b=a+r$, so their bundles generate the enlarged positive quotient. For $F\in\Db(I^{a,r}(d))$, pull back $F$ to this vector bundle, express it in the thick closure of these generators, and then pull back along the zero section. The zero section has finite Tor dimension. This expresses $F$ in the thick closure of the original large-window bundles. Essential surjectivity follows.

Only generation has been transported along these vector bundles; full faithfulness was established independently by \eqref{eq:positive-grade} and \eqref{eq:negative-grade}.
\end{proof}

\begin{proof}[Proof of Theorem~\ref{thm:main}]
Apply Proposition~\ref{prop:git-windows} to the large window in Theorem~\ref{thm:window-sod} and to each small tail window at parameter $\mu-e$. The genericity condition is unchanged by an integral shift. Transport the Hall functors and the semiorthogonal order through these equivalences. This gives \eqref{eq:main-sod} with the specified embeddings.
\end{proof}

\begin{remark}[What the window comparison does not assert]
For a vector-bundle projection $b:\Tot(E)\to X$, one has
\[
 \RHom_{\Tot(E)}(b^*F,b^*G)
   \simeq\RHom_X(F,G\otimes\operatorname{Sym}(E^\vee)).
\]
Thus pullback of a semiorthogonal decomposition need not remain semiorthogonal. The direct Hall argument in Section~\ref{sec:framed} is needed precisely to control these additional morphisms for the $r$ new arrows.
\end{remark}

\section{Trace potentials and matrix factorizations}\label{sec:mf}

Let $B$ be a smooth affine variety with trivial $G_d$-action. A family of cyclic quiver potentials with coefficients in $\cO(B)$ gives invariant functions $W_d$ on $B\times R^{a,r}_d$ for all $d$. Assume that
\begin{equation}\label{eq:potential-restriction}
 W_d\big|_{B\times\End(\C^d)^{\oplus3}}
                     =W_d^0:=\Tr C[A,B],\qquad W_0=0.
\end{equation}
The restriction in \eqref{eq:potential-restriction} means that all framing arrows are set equal to zero. No dependence on $B$ remains in this unframed part.

\begin{lemma}[Block additivity]\label{lem:potential}
On an attracting correspondence for ordered unframed blocks and a framed tail, one has
\begin{equation}\label{eq:potential-additive}
                  p^*W_d=q^*\left(\sum_iW_{d_i}^0+W_f\right).
\end{equation}
The same statement holds over $B$, with the base occurring once.
\end{lemma}
\begin{proof}
An invariant regular function has the same value at a point of the attracting locus and at its cocharacter limit. At the limit all the operators are block diagonal and framing maps occur only in the zero-weight tail. The trace of each unframed cyclic word is the sum of its traces on the diagonal blocks. A cyclic word passing through the framing vertex contributes only to the framed block. Equation~\eqref{eq:potential-additive} follows from \eqref{eq:potential-restriction}.
\end{proof}

\begin{remark}
The assertion is not true for an arbitrary invariant function specified only by a restriction to a smaller representation space. For example, a square of a trace can introduce products of traces from distinct diagonal blocks. The cyclic-potential hypothesis can instead be replaced by the explicit requirement \eqref{eq:potential-additive} for every Hall correspondence used below.
\end{remark}

We use the coherent, absolute dg category of matrix factorizations, with idempotent completion. For a window generated by equivariant vector bundles, its factorization window is the full subcategory represented by factorizations whose terms are finite sums and summands of those bundles; this is equivalently the corresponding window in the dg factorization category. For $W=0$ we retain the usual two-periodic category, rather than replacing it by the singularity category of an underived smooth zero fiber.

The enhancement used below is the $\Perf(\A^1)$-linear semiorthogonal-decomposition principle of Halpern-Leistner--Pomerleano \cite[Lemmas 3.8--3.9]{HLP}; see also \cite[Proposition 2.5]{PTI}. It applies to smooth finite-type stacks with affine stabilizers, so it does not require the unframed quotient stacks to be Deligne--Mumford. For a $\Perf(\A^1)$-linear component $\mathcal A$, the corresponding factorization component is
\[
 \C((\beta))\otimes_{\C[[\beta]]}
 \operatorname{Fun}^{\mathrm{ex}}_{\Perf(\A^1)}
                  (\Perf(\C),\mathcal A),
\]
where $\Perf(\C)$ is a module through the origin of $\A^1$. The zero fiber in this construction is the derived zero fiber; in particular the convention also covers $W=0$. For the affine quotient windows, this agrees with the bundle description above by \cite[Lemmas 2.3--2.4]{PTI}. The construction of factored derived functors from compatible resolutions is discussed in \cite[Sections 4--5]{BDFIK} and \cite[Section 3]{BFK}.

\begin{definition}
Let
\begin{align*}
 \mathrm{DT}^{a,r}_B(d,W)&=\MF(B\times I^{a,r}(d),W_d),\\
 \mathrm{PT}^{a,r}_B(d,W)&=\MF(B\times P^{a,r}(d),W_d),\\
 \cS(d)_w&=\MF(\cM(d)_w,W_d^0).
\end{align*}
In the formula below, the product of factorization categories is the Thom--Sebastiani, two-periodic dg tensor product, with its coherent idempotent completion. It is not the ordinary unperiodic tensor product. In the conventions of \cite[Proposition 3.17]{PTI}, the periodic parameter has degree $-2$ and the ungraded product is taken over $\C((\beta))$. A graded version uses the compatible graded Thom--Sebastiani product. If a geometric torus $T$ acts, all equivariant box products are relative over $BT$, that is, over $\Perf(BT)$ in the corresponding dg convention; the torus is shared, not duplicated among the factors.
\end{definition}

\begin{theorem}\label{thm:mf}
Assume \eqref{eq:generic} and \eqref{eq:potential-additive}. There is a semiorthogonal decomposition
\begin{equation}\label{eq:mf-sod}
 \mathrm{DT}^{a,r}_B(d,W)=
 \sod{\left(\mathop{\boxtimes}_{i=1}^k\cS(d_i)_{w_i}\right)
                         \boxt\mathrm{PT}^{a,r}_B(f,W)}
\end{equation}
with the indexing inequalities and order of Theorem~\ref{thm:main}. The embeddings are the factorized Hall functors with the tail lifted through the corresponding window. The equivariant version holds for an additional algebraic torus commuting with $G_d$ and preserving the correspondences, windows, and $W$. The usual graded version holds for a commuting grading torus of weight two on $W$, whose order-two element acts trivially on the representation space, and with grading-compatible correspondences and windows.
\end{theorem}
\begin{proof}
First take the external product of the ordinary window argument with $\Perf(B)$. The base is smooth and affine; the finite resolutions and weight-separation calculations remain valid after this flat scalar extension. No extra copy of $B$ is introduced for an unframed factor.

Set $X=B\times I^{a,r}(d)$ and let $\mathcal A_\gamma\subset\Perf(X)$ be the image of one component of the transported decomposition in Theorem~\ref{thm:main}. The source potential is the sum in \eqref{eq:potential-additive}. That identity and the projection formula make the Hall functor $\Perf(\A^1)$-linear. The source windows are stable under the action: a perfect complex on $\A^1$ is built by finite cones and summands from $\cO_{\A^1}$, and its pullback is built by the same operations using multiplication by the source potential. Restriction to a GIT locus is likewise $\Perf(\A^1)$-linear. Its inverse on a window therefore transports this module structure. It follows that the decomposition of $\Perf(X)=\Db(X)$ is $\Perf(\A^1)$-linear.

Apply \cite[Lemma 3.9]{HLP} to this decomposition. We obtain a semiorthogonal decomposition of $\MF(X,W_d)$ with the factorization components associated to $\mathcal A_\gamma$. The ordinary Hall equivalence onto $\mathcal A_\gamma$ identifies its source with that component as a $\Perf(\A^1)$-module category. Applying the exact-functor construction displayed above and then inverting $\beta$ preserves this equivalence. This proves full faithfulness at the factorization level for the window components themselves. Compatibility with the geometric factored pullback and pushforward follows from the equality of potentials on the correspondence and their Koszul resolutions.

The source factorization category is identified with the product in \eqref{eq:mf-sod} by Thom--Sebastiani. For the unframed windows this is the restriction to window factors in \cite[Proposition 3.17]{PTI}. Its comparison with the ambient unframed Hall decomposition also works after tensoring with the framed-tail Landau--Ginzburg category: the tail is left unchanged, and the restriction is made on the unframed factors. The tail-window equivalence is potential-linear by the preceding argument. The unframed cubic potential has only zero as a critical value, since its critical equations imply $[A,B]=0$. Hence there are no extra summands coming from nonzero critical values on an unframed factor.

For an additional torus preserving $W_d$, the argument is made in the corresponding equivariant module categories; all factors share that torus. For a grading torus as in the statement, use the graded linear version \cite[Proposition 2.5(b)]{PTI}. The hypothesis on the order-two element permits the usual half-weight grading convention, and all resolutions and adjunction maps are grading-compatible. This gives the equivariant and graded assertions.
\end{proof}

\section{Two algebraic local models}\label{sec:models}

\subsection{An open chart of the framed conifold quiver}
Let $U$ and $V$ have dimensions $r$ and $n$, and let $V_1$ have dimension $n+r$. Consider arrows
\[
 u:U\to V_1,\qquad a_1,a_2:V_1\to V,\qquad b_1,b_2:V\to V_1.
\]
The group $\GL(V_1)\times\GL(V)$ acts, while $U$ is fixed. Let $\mathcal C^\circ_{r,n}$ be the open quotient stack on which
\begin{equation}\label{eq:conifold-open}
                         (u,b_2):U\oplus V\xrightarrow{\sim}V_1.
\end{equation}
The conifold trace potential is
\begin{equation}\label{eq:conifold-original}
           W_{\mathrm{con}}=\Tr_V(a_1b_1a_2b_2-a_1b_2a_2b_1).
\end{equation}

\begin{proposition}[The conifold chart]\label{prop:conifold}
There is an isomorphism of quotient stacks
\begin{equation}\label{eq:conifold-chart}
 \mathcal C^\circ_{r,n}\simeq
 \left[\frac{\Hom(U,V)^{\oplus2}\oplus\Hom(V,U)
                             \oplus\End(V)^{\oplus3}}{\GL(V)}\right]
                   \simeq\cX^{r,r}_n.
\end{equation}
Under this isomorphism the potential is
\begin{equation}\label{eq:conifold-reduced}
           W_{r,n}=\Tr_U(vAu_1-vBu_2)+\Tr_V C[A,B].
\end{equation}
The assertions hold in families, and hence also identify the corresponding classical critical stacks.
\end{proposition}
\begin{proof}
The isomorphism \eqref{eq:conifold-open} canonically identifies $V_1$ with $U\oplus V$ after quotienting by $\GL(V_1)$. In this identification write
\[
 u=\begin{pmatrix}1\\0\end{pmatrix},\quad
 b_2=\begin{pmatrix}0\\1\end{pmatrix},\quad
 a_1=(u_1,B),\quad a_2=(u_2,A),\quad
 b_1=\begin{pmatrix}v\\C\end{pmatrix}.
\]
The residual group is $\GL(V)$, with its usual action on all displayed maps. This construction uses only an isomorphism of vector bundles, so commutes with base change and proves the stack assertion, not merely a bijection of closed points.

Multiplication of these block matrices gives
\begin{align*}
 W_{\mathrm{con}}
 &=\Tr_V\bigl((u_1v+BC)A-B(u_2v+AC)\bigr)\\
 &=\Tr_U(vAu_1-vBu_2)+\Tr_V C(AB-BA).
\end{align*}
This proves \eqref{eq:conifold-reduced} and the assertion about critical stacks.
\end{proof}

For example, the critical equations of \eqref{eq:conifold-reduced} are
\begin{align*}
 [A,B]&=0,& vA&=0,&vB&=0,\\
 Au_1-Bu_2&=0,&[B,C]+u_1v&=0,&[C,A]-u_2v&=0.
\end{align*}
The determinant chambers on the right side of \eqref{eq:conifold-chart} satisfy Theorem~\ref{thm:mf} with $a=r$.

\begin{remark}[The dimension vector and stability are part of the statement]
This chart uses dimension vector $(r,n+r,n)$ and the isomorphism condition \eqref{eq:conifold-open}. It is not the entire injective-$b_2$ locus for dimension vector $(r,n+1,n)$. Nor do we identify its determinant chambers with a specified chamber of perverse coherent systems on the resolved conifold. Such an identification would require the corresponding tilting equivalence, numerical class, and stability comparison to be checked separately. The algebraic chart and its potential do not require those additional claims.
\end{remark}

\subsection{An algebraically defined extended ADHM model}
Fix a smooth affine parameter variety $B$ and a polynomial
\begin{equation}\label{eq:f-polynomial}
             f_\alpha(A,B_0)=1+\sum_{i,j}\alpha_{ij}A^i B_0^j,
\end{equation}
where $\alpha_{ij}\in\cO(B)$, the sum is finite, and the order $A^iB_0^j$ is part of the definition. Here $B_0$ is a loop operator, distinguished from the parameter variety $B$.

On the smooth space
\[
 S=\bigl(\Hom(U,V)\oplus\Hom(V,U)\oplus\End(V)^{\oplus2}\bigr)\times B
\]
with matrix coordinates $(u_1,v,A,B_0)$ and parameter in $B$, consider the equivariant section of the trivial bundle with fiber $\End(V)\oplus\Hom(V,U)$ given by
\begin{equation}\label{eq:ADHM-section}
               s=([A,B_0]+u_1v,\ v f_\alpha(A,B_0)).
\end{equation}
Define the derived stack
\begin{equation}\label{eq:ADHM-derived}
                     \mathfrak Z_{r,n}(f)=[Z^{\mathrm{der}}(s)/\GL(V)].
\end{equation}
This is a definition by a derived zero locus. It does not refer to a moduli space of pairs on a surface.

\begin{proposition}[Multiplier potential]\label{prop:ADHM}
Let $C\in\End(V)$ and $u_2\in\Hom(U,V)$ be the dual fiber coordinates. On $B\times R^{r,r}_n$, define
\begin{equation}\label{eq:ADHM-potential}
 W_f=\Tr_V C([A,B_0]+u_1v)
                        +\Tr_U\bigl(v f_\alpha(A,B_0)u_2\bigr).
\end{equation}
There is a canonical isomorphism of classical stacks
\begin{equation}\label{eq:ADHM-cotangent}
 t_0\bigl(T^*[-1]\mathfrak Z_{r,n}(f)\bigr)
       \simeq[\Crit(W_f)/\GL(V)].
\end{equation}
Here $T^*[-1]$ is absolute over $\C$, and the critical equations include derivatives in the parameter directions. The potential satisfies the block-additivity hypothesis of Theorem~\ref{thm:mf}.
\end{proposition}
\begin{proof}
For a section $s$ of a vector bundle $E$ on a smooth affine space $S$, a Koszul model for its derived zero locus is
\[
                         (\cO_S\otimes\Lambda^\bullet E^\vee,d_s),
\]
with $E^\vee$ in degree $-1$. In the shifted cotangent construction the dual obstruction variables are ordinary fiber coordinates $\xi$ on $E^\vee$. The degree-zero equations are
\[
                            s=0,\qquad (ds)^*\xi=0.
\]
These are precisely the critical equations of the multiplier function $\langle\xi,s\rangle$ on $\Tot(E^\vee)$. Passing to a reductive quotient gives the same assertion for classical truncations; the additional Lie-algebra terms in the derived presentation do not add degree-zero equations. In our case the trace pairings identify the dual variables with $(C,u_2)$ and give \eqref{eq:ADHM-potential}. Since the construction is absolute, $ds$ includes differentiation in all tangent directions of the parameter variety $B$.

Each term of \eqref{eq:ADHM-potential} is a trace of a cyclic word of the framed quiver, with coefficients in $\cO(B)$. Setting the framing maps to zero leaves $\Tr_V C[A,B_0]$. Lemma~\ref{lem:potential} applies.
\end{proof}

\begin{remark}
The proposition gives a shifted-cotangent model and two quiver factorization categories. It does not identify them with the reduced-curve pair moduli problem of \cite[Section~2]{Chuang2023v2}, in which the curve multiplicity is fixed independently of the framing rank. Such a geometric identification is not asserted in this paper.
\end{remark}

\section{Points on affine three-space and a nontrivial rank-two example}\label{sec:examples}

\subsection{The degree-zero geometric application}
Set $a=0$. Then $R^{0,r}_d=\Hom(\C^r,V)\oplus\End(V)^{\oplus3}$, and consider
\[
                                  W_d^0=\Tr C[A,B].
\]

\begin{proposition}\label{prop:quot}
There is a scheme-theoretic isomorphism
\begin{equation}\label{eq:quot-critical}
 \Crit(W_d^0:I^{0,r}(d)\to\A^1)
                     \simeq\Quot_{\A^3}(\cO_{\A^3}^{\oplus r},d).
\end{equation}
\end{proposition}
\begin{proof}
The critical equations with respect to $A,B,C$ are
\[
                           [B,C]=[C,A]=[A,B]=0.
\]
The potential is independent of the framing map. Positive stability says that its $r$ vectors generate $V$ under the three operators. A commuting triple therefore equips $V$ with the structure of a finite-dimensional $\C[x,y,z]$-module and the framing gives a surjection
\[
                               \C[x,y,z]^{\oplus r}\twoheadrightarrow V.
\]
Conversely, such a quotient gives the commuting multiplication operators and the framing vectors. In families, a length-$d$ quotient has a locally free pushforward of rank $d$ over the base, and the same construction gives commuting endomorphisms of that vector bundle. Choosing its frame is exactly the principal $G_d$-bundle removed by the quotient. Since the Jacobian ideal is the ideal of the three commutators, the correspondence identifies the schemes and their functors of families, not just their closed points. This critical presentation was established by Beentjes and Ricolfi \cite{BR}; its $d$-critical structure is compared with the structure from derived deformation theory in \cite{RS}.
\end{proof}

For $d>0$ the negative chamber is empty. Theorem~\ref{thm:mf} consequently has only unframed factors in this case. Thus the decomposition applies to the category associated to the critical presentation \eqref{eq:quot-critical}.

A volume-preserving geometric torus $T=(\C^*)^2$ acts by
\begin{equation}\label{eq:torus}
 (A,B,C)\longmapsto(t_1A,t_2B,(t_1t_2)^{-1}C),
\end{equation}
with trivial action on the framing. Independently, a grading torus can act by $C\mapsto z^2C$ and fix $A,B$ and the framing, giving $W_d^0$ grading weight two. These are different actions and should not be identified.

\subsection{The complete calculation for \texorpdfstring{$d=1$}{d=1}}
Put $p=a+r$ and $q=a$. Since the loop operators are scalars, the two ordinary quotients are
\begin{align}
 I^{a,r}(1)&=\A^3\times\Tot_{\PP^{p-1}}\bigl(\cO(-1)^{\oplus q}\bigr),\label{eq:d1-I}\\
 P^{a,r}(1)&=\A^3\times\Tot_{\PP^{q-1}}\bigl(\cO(-1)^{\oplus p}\bigr)
                                             \quad(q>0).\label{eq:d1-P}
\end{align}
For $q=0$ the second quotient is empty. Let $m$ be the smallest integer strictly larger than $L=-r-\mu-a/2$. Since the intervals have integral lengths, the large-window weights are
\[
                         m,m+1,\ldots,m+p-1,
\]
whereas the small-window weights are $m+r,\ldots,m+p-1$. The wall weights are $m,\ldots,m+r-1$.

Let $j:\A^3\times\PP^{p-1}\hookrightarrow I^{a,r}(1)$ be the zero section. Theorem~\ref{thm:main} becomes
\begin{equation}\label{eq:d1-sod}
 \Db(I^{a,r}(1))=
 \sod{j_*(\cO(m)\boxt\Db(\A^3)),\ldots,
 j_*(\cO(m+r-1)\boxt\Db(\A^3)),\Phi\Db(P^{a,r}(1))}.
\end{equation}
In this expression the line bundle is on $\PP^{p-1}$; the notation suppresses the harmless order of the two geometric factors.

Here one can verify the categorical assertions without the general proof. The normal bundle of $j$ is $N=\cO(-1)^{\oplus a}$. The self-intersection filtration computing morphisms between two zero-section objects has graded pieces involving
\[
                           H^*(\PP^{p-1},\cO(w-w'-t)),\quad0\leq t\leq a.
\]
For $w=w'$, all terms with $t>0$ vanish. For two wall weights with $w'>w$, all these twists lie between $-p+1$ and $-1$, so every term vanishes. Morphisms from a small-window generator $\cO(k)$ to $j_*\cO(w)$ similarly reduce to $H^*(\PP^{p-1},\cO(w-k))$ in the same vanishing range. Finally, the Koszul resolution of $j_*\cO(w)$ uses $\cO(w),\ldots,\cO(w+a)$. Starting with the small-window generators and descending through the wall weights recovers all large-window generators. This proves full faithfulness, semiorthogonality, and generation in \eqref{eq:d1-sod} directly.

\begin{example}[A wall with $a>0$ and $r=2$]\label{ex:blowup}
Take $a=1$, $r=2$, and $\mu=1/7$. The large-window weights are $-2,-1,0$, the small-window weight is $0$, and the two wall weights are $-2,-1$. The quotients are
\[
 I^{1,2}(1)=\A^3\times\Tot_{\PP^2}\cO(-1)
                  \simeq\operatorname{Bl}_{\A^3}\A^6,
 \qquad P^{1,2}(1)\simeq\A^6.
\]
The center of the blowup is $\A^3\times\{0\}\subset\A^3\times\A^3$. Formula \eqref{eq:d1-sod} gives two fully faithful copies of $\Db(\A^3)$ and one copy of $\Db(\A^6)$, with the exceptional pieces twisted by $\cO(-2)$ and $\cO(-1)$. Thus the higher-framing calculation is not confined to the case of an empty PT chamber.
\end{example}

For $a=0$, $d=1$, one instead obtains $I^{0,r}(1)=\A^3\times\PP^{r-1}$ and $W_1^0=0$. The $r$ wall factors reproduce the usual length-$r$ exceptional collection on the projective-space factor, tensored with $\Db(\A^3)$, or with its appropriate periodic or graded version.

\section{A formal MacMahon identity and its categorical interpretation}\label{sec:macmahon}

Write $p_2(n)$ for the number of ordinary partitions of $n$, with $p_2(0)=1$, and set
\[
       \sum_{n\geq0}p_2(n)q^n=\prod_{k\geq1}(1-q^k)^{-1},\qquad
       M(q)=\prod_{k\geq1}(1-q^k)^{-k}.
\]
The following identity is independent of categorical $K$-theory.

\begin{proposition}\label{prop:MacMahon}
As an identity of formal power series,
\begin{equation}\label{eq:MacMahon}
 \prod_{\substack{0\leq p/b<r\\(p,b)=1,\ b\geq1}}
             \prod_{k\geq1}(1-q^{bk})^{-1}=M(q)^r.
\end{equation}
Equivalently, the sum over finite strictly increasing slope tuples
$0\leq v_1/d_1<\cdots<v_k/d_k<r$ of the terms
\[
                 \prod_{i=1}^k p_2(\gcd(d_i,v_i))\,q^{\sum_i d_i}
\]
is $M(q)^r$, with the empty tuple contributing $1$.
\end{proposition}
\begin{proof}
Fix a reduced slope $p/b$. A block of this slope has $(d,v)=(mb,mp)$, so its multiplicity is $p_2(m)$. Summing over the possibility of omitting the slope or choosing such a block gives
\[
                       \sum_{m\geq0}p_2(m)q^{mb}
                            =\prod_{k\geq1}(1-q^{bk})^{-1}.
\]
There are $r\varphi(b)$ reduced fractions of denominator $b$ in $[0,r)$, where $\varphi(1)=1$ and $0/1$ is included. Consequently, the exponent of $(1-q^n)^{-1}$ in the left side of \eqref{eq:MacMahon} is
\[
                          r\sum_{b\mid n}\varphi(b)=rn.
\]
This gives \eqref{eq:MacMahon}. Each coefficient involves only denominators $b\leq n$, so all products are well-defined formally.
\end{proof}

To compare this with the $a=0$ categorical decomposition in a fixed degree $d$, choose $\mu=-r+\varepsilon$ with $0<\varepsilon<1/d$ and satisfying genericity. Every slope has denominator at most $d$, so
\[
                   -\varepsilon<v_i/d_i<r-\varepsilon
               \quad\Longleftrightarrow\quad 0\leq v_i/d_i<r.
\]
This is a coefficientwise choice of window; a single fixed positive $\varepsilon$ does not give that equivalence in every degree. Also $\gcd(d_i,v_i)=\gcd(d_i,w_i)$, because $v_i-w_i$ is an integral multiple of $d_i$.

We next justify the localized categorical interpretation by using the specific product formula available in point theory. This is stronger than merely knowing the single-block ranks, and is not a general K\"unneth assertion.

Let $T$ be the torus in \eqref{eq:torus}, let $K=\Z[t_1^{\pm1},t_2^{\pm1}]$, and let $F=\Q(t_1,t_2)$. Denote by $\mathcal T_T(n)_v$ the commuting-stack window category of \cite[Section 4.2]{PTI}. Equivariant dimensional reduction and its compatibility with these windows give
\begin{equation}\label{eq:dimensional-reduction}
                 \mathcal T_T(n)_v\simeq\cS_T^{\gr}(n)_v;
\end{equation}
see \cite[Theorem 4.4 and Lemma 4.5]{PTI}.

\begin{lemma}[The strict-slope localized product input]\label{lem:K-product}
For integers $(d_i,v_i)$ with $d_i>0$ and strictly increasing slopes $v_i/d_i$, one has
\begin{equation}\label{eq:K-product}
 K_0\!\left(\mathop{\boxtimes}_{BT,i}\mathcal T_T(d_i)_{v_i}\right)
                    \otimes_K F
 \simeq\bigotimes_{i,F}\left(K_0(\mathcal T_T(d_i)_{v_i})\otimes_K F\right).
\end{equation}
Each single-factor dimension is $p_2(\gcd(d_i,v_i))$.
\end{lemma}
\begin{proof}
These are the localized product statement in the proof of \cite[Corollary 4.14]{PTI} and the single-block basis theorem \cite[Theorem 4.13]{PTI}. More precisely, the product statement compares the strictly ordered slope summands of the commuting-stack semiorthogonal decomposition, equation (4.37) there, with its shuffle-algebra basis, equation (4.36). The products of the single-slope basis elements form a basis of each ordered summand, proving \eqref{eq:K-product}. This argument assumes strictly increasing slopes, but does not restrict them to $[0,1)$; that shorter interval is used only in the subsequent rank-one DT specialization in that corollary. The single-slope basis is indexed by ordinary partitions of $\gcd(d_i,v_i)$.
\end{proof}

\begin{corollary}[Localized rank formula for the point model]\label{cor:K-MacMahon}
Use the geometric torus \eqref{eq:torus} and the independent grading with $C$ of weight two. Set
\[
 \mathrm{DT}^{\gr}_{r,T}(d)=\MF_T^{\gr}(I^{0,r}(d),W_d^0),\qquad
 a_{r,d}=\dim_F\bigl(K_0(\mathrm{DT}^{\gr}_{r,T}(d))\otimes_K F\bigr).
\]
Then $a_{r,d}$ is finite, $a_{r,0}=1$, and
\begin{equation}\label{eq:categorical-MacMahon}
                              \sum_{d\geq0}a_{r,d}q^d=M(q)^r.
\end{equation}
\end{corollary}
\begin{proof}
Choose the window coefficientwise as above and apply the graded $T$-equivariant version of Theorem~\ref{thm:mf}. All components are products of $\cS_T^{\gr}(d_i)_{w_i}$, relative over $BT$. Write $v_i$ as in \eqref{eq:vi-intro}, with $f=0$. Since $v_i-w_i$ is divisible by $d_i$, a determinant twist on the $i$th block, followed by \eqref{eq:dimensional-reduction}, gives an equivalence of its product with the product of $\mathcal T_T(d_i)_{v_i}$. The slopes $v_i/d_i$ are strictly increasing, so Lemma~\ref{lem:K-product} makes the rank of that component equal to $\prod_i p_2(\gcd(d_i,v_i))$.

Additivity of $K_0$ for the finite semiorthogonal decomposition gives precisely the coefficient sum in Proposition~\ref{prop:MacMahon}. This proves \eqref{eq:categorical-MacMahon} and finiteness. In degree zero the category has the single point contribution and its localized rank is one.
\end{proof}

\begin{remark}
The formal identity in Proposition~\ref{prop:MacMahon} is independent of the Hall theorem. Corollary~\ref{cor:K-MacMahon} additionally uses the factorization decomposition, equivariant dimensional reduction, and the specific strict-slope product formula \eqref{eq:K-product}. It does not infer multiplicativity of localized $K$-theory from semiorthogonality alone, and makes no such claim for arbitrary potentials or arbitrary products of categories.
\end{remark}

\end{document}